\documentclass[envcountsect, envcountsame, graybox]{svmult}

\usepackage{mathptmx}       
\usepackage{helvet}         
\usepackage{courier}        
\usepackage{type1cm}        
\usepackage{makeidx}         
\usepackage{graphicx}        
\usepackage{multicol}        
\usepackage[bottom]{footmisc}

\makeindex             

\usepackage{amstext,amsfonts,amssymb,amscd,amsbsy,amsmath}
\usepackage{tikz-cd}	
\usepackage{tikz}
\usepackage{ytableau}
\usepackage{caption}

\smartqed

\DeclareMathOperator{\Gr}{Gr}

\DeclareMathOperator{\Hom}{Hom}
\DeclareMathOperator{\im}{im}

\begin{document}

\title*{Specht Polytopes and Specht Matroids}
\author{John D. Wiltshire-Gordon, Alexander Woo, and Magdalena Zajaczkowska}
\institute{John D. Wiltshire-Gordon \at University of Wisconsin, Madison; Department of Mathematics;
Van Vleck Hall;
480 Lincoln Drive;
Madison, WI 53706; USA, 
 \email{jwiltshiregordon@gmail.com} \\
 Alexander Woo \at University of Idaho; Department of Mathematics; University of Idaho; 875 Perimeter Drive MS 1103; Moscow, ID 83844; USA, 
\email{awoo@uidaho.edu} \\ 
Magdalena Zajaczkowska \at University of Warwick; Mathematics Institute; Gibbet Hill Rd; Coventry CV4 7AL;
UK, 
\email{m.a.zajaczkowska@warwick.ac.uk}}
%
%
\maketitle

\abstract*{The generators of the classical Specht module satisfy intricate relations.
  We introduce the Specht matroid, which keeps track of these
  relations, and the Specht polytope, which also keeps track of convexity relations.
  We establish basic facts about the Specht polytope, for example, that the symmetric
  group acts transitively on its vertices and irreducibly on its ambient real vector space.
  A similar construction builds 
  a matroid and polytope for a tensor product of Specht modules, 
  giving ``Kronecker matroids'' and ``Kronecker polytopes'' instead 
  of the usual Kronecker coefficients. We dub this process of 
  upgrading numbers to matroids and polytopes ``matroidification,''
  giving two more examples.   In the course of describing these objects, we also give
  an elementary account of the construction of Specht modules different from the standard one.
  Finally, we provide code to compute with Specht matroids and their Chow rings.
   }

\abstract{The generators of the classical Specht module satisfy intricate relations.
  We introduce the Specht matroid, which keeps track of these
  relations, and the Specht polytope, which also keeps track of convexity relations.
  We establish basic facts about the Specht polytope, for example, that the symmetric
  group acts transitively on its vertices and irreducibly on its ambient real vector space.
  A similar construction builds 
  a matroid and polytope for a tensor product of Specht modules, 
  giving ``Kronecker matroids'' and ``Kronecker polytopes'' instead 
  of the usual Kronecker coefficients. We dub this process of 
  upgrading numbers to matroids and polytopes ``matroidification,''
  giving two more examples.   In the course of describing these objects, we also give
  an elementary account of the construction of Specht modules different from the standard one.
  Finally, we provide code to compute with Specht matroids and their Chow rings.
   }

\hspace{.2in}

The irreducible representations of the symmetric group $S_n$ were worked out by Young and Specht in the early 20th century, and they remain omnipresent in algebraic combinatorics.  The symmetric group $S_n$ has a unique irreducible representation for each partition of $n$.  For example, $S_4$ has exactly five irreducible representations corresponding to the partitions
$$
(4) \;\;\;\;\; (3,1) \;\;\;\;\; (2,2) \;\;\;\;\; (2,1,1) \;\;\;\;\; (1,1,1,1).
$$
Young and Specht constructed these irreducible representations, which are now called \emph{Specht modules}.  Young~\cite{Young} gave a matrix representation and Specht~\cite{Specht} gave a combinatorial spanning set.  Garnir~\cite{Garnir} later explained how to rewrite Specht's spanning set in terms of Young's basis.  These rewriting rules are now called  \emph{Garnir relations}.  Modern accounts of these constructions can be found in Sagan~\cite{Sagan} or James and Kerber~\cite{JK}.

This classical approach, however, privileges Young's basis over
other bases and the Garnir relations over other linear dependencies.
Focusing on Young's basis and the Garnir relations immediately breaks
the symmetry of Specht's spanning set and
ignores its other combinatorial properties.  Certainly, there are linear relations
other than those given by Garnir and bases other than those given by Young to investigate!  

To this end, we introduce the \emph{Specht matroid}, which encodes all the
linear dependencies among the vectors of Specht's spanning set.  We
also introduce the \emph{Specht polytope}, which provides a way to
visualize the Specht module, since the polytope sits inside the corresponding real vector space
with positive volume, and the action of the symmetric group takes the
polytope to itself.

In the case of the partition $(n-1,1)$, we recover both classical
constructions and objects of current research interest.  The
corresponding Specht polytope is essentially the \emph{root
  polytope} of type $A_n$.  In Theorem~\ref{theorem:root_polytope}, we
record a result of Ardila, Beck, Hosten, Pfeifle and
Seashore~\cite{ABHPS} describing the faces of this polytope.  The
Specht matroid for the partition $(n-1,1)$ is the matroid of the braid
arrangement, and hence its Chow ring is the cohomology ring for the
moduli space $\overline{\mathcal{M}_{0,n}}$ of $n$ marked points on
the complex projective line~\cite{DP,FY}.

We compute further examples of Chow rings in \S\ref{section:chow},
including the solution to Problem~1 on Grassmannians in \cite{Sturmfels}, which partially inspired
this project.
We state a combinatorial conjecture for the graded dimensions of the
Chow rings for the Specht matroid for the partition $(2, 1^{n-1})$.
However, we do not study any further connections with moduli spaces.

Our approach allows us to 
upgrade familiar combinatorial coefficients to matroids and polytopes.
By analogy with \emph{categorification}, which sometimes upgrades numbers to vector spaces,
we call this process \emph{matroidification}, or \emph{polytopification} when working over the reals.
  This is the subject of Section~\ref{section:upgrading}.
In Theorem~\ref{theorem:kronecker}, 
we polytopify the Kronecker coefficients, building the \emph{Kronecker polytopes}.
We also construct the
\emph{Kronecker matroids} encoding the Garnir-style rewriting rules
that govern linear dependence in a tensor product of Specht modules.
An analogue of Young's basis for the Kronecker matroid would provide a
combinatorial rule for computing Kronecker coefficients.  In
Theorems~\ref{theorem:littlewood-richardson} and~\ref{theorem:plethysm}, we 
give similar results for Littlewood-Richardson coefficients and 
plethysm coefficients.





\bigskip  
The outline of the article is as follows. We begin in Section~\ref{section:introduction} with a self-contained
construction of the Specht module that is suited to our purposes.
This construction is a bit unusual in that it makes no mention of
tabloids, polytabloids, standard tableaux, or the group algebra.  In
Sections~\ref{section:spechtmatroids}, \ref{section:chow}, and
\ref{section:spechtpolytopes}, we introduce respectively the Specht
matroid, its Chow ring, and the Specht polytope; we then
prove some basic general facts about them.  Section
\ref{section:examples} is devoted to the partitions $(n-1,1)$ and
$(2,1^{n-1})$, for which the Specht matroids and polytopes coincide
with other well-studied objects.  We describe Kronecker,
Littlewood--Richardson, and plethysm matroids and polytopes in Section
\ref{section:upgrading}.  Section~\ref{section:code} includes code for
calculating the objects described in this paper.

\section{Introduction to Specht Modules} \label{section:introduction}
Our aim in this section is to give an exposition of the representation
theory of the symmetric group that is motivated from elementary
combinatorial considerations.  The reader who wishes to start with the
main statements should first look at
Definitions~\ref{definition:YoungsChar} and~\ref{SpechtMatrix} and
Theorem~\ref{theorem:SpechtIrred}.

We begin with an elementary combinatorics problem: In how many distinct ways can the letters in a word \texttt{TENNESSEE} be rearranged?  There are $9!$ ways to move the letters around, but since some letters are repeated, some of these rearrangements give the same string.  For example, the four \texttt{E}s can be rearranged to appear in any order without affecting the string.  This reasoning gives the following answer:
$$
\# \{\mbox{ rearrangements of \texttt{TENNESSEE} }\} = \frac{9!}{1! \cdot 4! \cdot 2! \cdot 2!}.
$$
The idea of rearranging letters can be formalized as an action of the symmetric group. In our example $S_9$ acts on the word TENNESSEE. The stabilizer subgroup of $S_9$ with respect to the word TENNESSEE is isomorphic to $S_1\times S_4\times S_2\times S_2$. Hence the previous argument actually provides an isomorphism of \emph{$S_{9}$-sets}, which is to say, a bijection that commutes with the group action: 
$$
\{\mbox{ rearrangements of \texttt{TENNESSEE} }\} \simeq \frac{S_{9}}{S_1 \times S_4 \times S_2 \times S_2}.
$$
Using the orbit-stabilizer theorem, we recover the numerical answer above.

Now we add a layer of complexity.  Suppose we wish to understand the $S_9$-set
$$
\begin{array}{c}
 \{\mbox{ rearrangements of \texttt{TENNESSEE} }\} \phantom{,}\\
 \times \\
 \{\mbox{ rearrangements of \texttt{SASSAFRAS} }\} , \\
\end{array} 
$$
where $S_9$ acts diagonally (i.e. simultaneously) on the two factors. This diagonal action makes sense because \texttt{SASSAFRAS} has the same number of letters as \texttt{TENNESSEE}. Each factor in this Cartesian product has a single $S_9$-orbit, but the product certainly does not!    For example, the pairs
$$
\left(
\begin{array}{c}
 \texttt{EEEENNSST} \\
 \texttt{SSSSAAAFR} \\
\end{array}
\right)
\;\;\; \mbox{and} \;\;\;
\left(
\begin{array}{c}
 \texttt{TENNESSEE} \\
 \texttt{SASSAFRAS} \\
\end{array}
\right)
$$
cannot be in the same orbit because the upper \texttt{E}s and lower \texttt{S}s always appear together in the first pair but not in the second.  Another proof is that their orbits have different sizes.  Indeed, if we consider a column such as $\genfrac{}{}{0pt}{}{\texttt{E}}{\texttt{S}}$ as a compound letter, then the total number of distinct rearrangements of the first compound word in these compound letters is $9!/(4! \cdot 2! \cdot 1! \cdot 1! \cdot 1!) = 7560$, which does not equal the number $9!/(1! \cdot 3! \cdot 2! \cdot 1! \cdot 1! \cdot 1!) = 30240$ of rearrangements of the second compound word.

For the construction of the Specht module, we are interested in \emph{free orbits}.  (Recall that an orbit is free if each of its points has trivial stabilizer.)  In our context, a non-trivial stabilizer comes from repeated columns.  So a pair is in a free orbit if and only if all of its columns are distinct.  For example,
$$
\left(
\begin{array}{c}
 \texttt{SNETNEESE} \\
 \texttt{SASSSAFAR} \\
\end{array}
\right)
$$
has no repeated columns, so its orbit is free.  We claim:
\begin{itemize}
\item there is only one free orbit, and therefore,
\item we have already found it.
\item The proof is basically a picture, and even better,
\item the proof-picture-idea is strong enough to construct a complete set of irreducible representations over $\mathbb{C}$ for the symmetric group $S_n$.  These are the Specht Modules.  (The story would be the same for any field of characteristic zero.)
\end{itemize}
Here is the picture.

\begin{figure}[!ht] 
\centering
\includegraphics{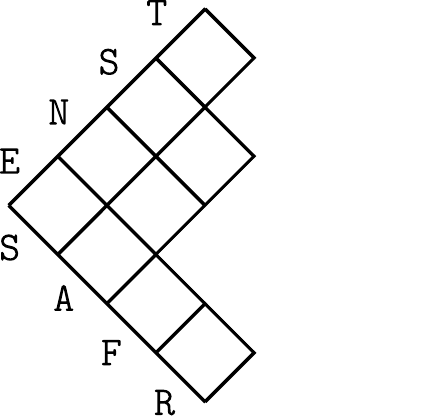}
\caption{A simultaneous histogram}
\label{figure:simultaneous-histogram}
\end{figure}

The boxes provide a simultaneous histogram tallying the letter multiplicities for each word.  From the picture, we see that the letter \texttt{E} from the word \texttt{TENNESSEE} must appear once with each of the letters \texttt{S}, \texttt{A}, \texttt{F}, and \texttt{R}.  Indeed, in order to keep distinct the four columns in which \texttt{E} appears, \texttt{E} must be paired with each of the four available letters in the bottom row.

Removing the four \texttt{E}s from the pool along with one copy of each of the letters \texttt{S}, \texttt{A}, \texttt{F}, and \texttt{R}, we may proceed to pair off \texttt{N} with the two letters \texttt{S} and \texttt{A}.  Continuing inductively, we see that the combinations that appear in a valid pair of rearrangements give the boxes in the diagram.

We give some definitions that encode these pictures.
\begin{definition}
A \emph{partition} of $n\in \mathbb{N}$ is an integer vector $\lambda=(\lambda_1,\ldots,\lambda_\ell)$ 
such that $\lambda_1\geq\lambda_2\geq\cdots\geq\lambda_\ell>0$ with $\lambda_1+\cdots+\lambda_\ell=n$.
The number $\ell=\ell(\lambda)$ is the \emph{length} of the partition.
\end{definition}
\noindent

\begin{definition}
A \emph{diagram} is a finite subset of $\mathbb{N}_{\geq 1}^2$.  The elements of a diagram are called \emph{boxes}.
\end{definition}

\begin{definition}
Given a partition $\lambda$, the \emph{diagram associated to $\lambda$} is
$$D(\lambda)=\{(i,j)\mid 1\leq j\leq \lambda_i\},$$
where, by convention, $\lambda_i=0$ if $i>\ell(\lambda)$.
\end{definition}
\noindent
The partition in our running example is $(4, 3, 1, 1)$; its associated diagram is
$$
 \{(1,1), (1,2), (1,3), (1,4), (2,1), (2,2), (2,3), (3,1), (4,1) \}.
$$
Here and everywhere else, we will use matrix coordinates, so $(2,3)$ denotes the box in the second row and third column.

The following proposition is immediate.

\begin{proposition}
A diagram $D$ is the diagram of some partition $\lambda$ if and only if $D$ is closed under coordinate-wise $\leq$.  In other words, $D=D(\lambda)$ for some $\lambda$ if and only if, for any $(i,j)\in D$ and any $(a,b)$ with $1\leq a\leq i$ and $1\leq b\leq j$, $(a,b)\in D$.
\end{proposition}

\begin{definition}\label{def:complementary_rearrangements}
We say two words $w_1, w_2$ of (the same) length $n$ \emph{have complementary rearrangements} if the diagonal action of $S_n$ on the product
$$
\{ \mbox{ rearrangements of $w_1$ } \} \times \{ \mbox{ rearrangements of $w_2$ } \} 
$$
has a unique free orbit.  If, furthermore, $(w_1, w_2)$ is in this free orbit, then we say $w_1$ and $w_2$ are \emph{complementary}.
\end{definition}

For example, our diagram above shows that  \texttt{TENNESSEE} and \texttt{SASSAFRAS} have complementary rearrangements.  The two words are not complementary, but \texttt{TENENEESS} and \texttt{SASSAFRAS} are complementary.

\begin{theorem} \label{theorem:diagrams}
Two words $w_1, w_2$ have complementary rearrangements if and only if there exists a parititon diagram $D$ with the ``simultaneous histogram'' property
\begin{align*}
\mbox{ $\#$occurrences in $w_i$ of its } & j^{\scriptscriptstyle \mbox{th}} \mbox{-most-common letter}  = \# \{ (d_1, d_2) \in D \,|\, d_i = j \}.
\end{align*}
\end{theorem}

\begin{proof}
We have already argued the hard direction.  If we have a partition diagram $D$ with the simultaneous histogram property, we can put in the box $(d_1,d_2)$ the $d_1$-th most common letter in $w_1$ and the $d_2$-th most common letter in $w_2$ (breaking ties arbitrarily).  The boxes have distinct pairs, so there exists at least one free orbit; this orbit is unique by the iterative argument below Figure~\ref{figure:simultaneous-histogram}.  In the other direction, rewrite the words using numbers in $\mathbb{N}_{\geq 1}$ so that (in each word) $k$ appears at least as often as $k+1$.  Then take 
$$
D = \{(d_1, d_2) \; | \; d_1 \mbox{ appears in a column with } d_2 \mbox{ in the unique free orbit} \}.
$$
\qed
\end{proof}

We will proceed to use the idea of complementary words to construct
irreducible representations of the symmetric group $S_n$.  Before
doing so, we recall some basic definitions in representation theory.

\begin{definition}
  Let $G$ be a group.  A \emph{(complex) representation} of $G$ is a
  $\mathbb{C}$-vector space $V$ along with a linear action of $G$ on
  $V$, meaning that, for any vectors $v,w\in V$, any scalar
  $c\in\mathbb{C}$, and any group element $g\in G$,
\begin{itemize}
\item $g\cdot v + g \cdot w = g\cdot (v+w)$, and
\item $c(g\cdot v) = g\cdot (cv)$.
\end{itemize}
Alternatively, the data of a representation can be encoded in a group
homomorphism $G\rightarrow GL(V)$, where $GL(V)$ is the group of
invertible linear automorphisms of $V$.
\end{definition}

\begin{definition}
  If $V$ and $W$ are representations of $G$, a linear transformation
  $\varphi: V\rightarrow W$ is a \emph{map of $G$ representations} if
  $\varphi$ commutes with the action of $G$, meaning that
  $\varphi(g\cdot v) = g\cdot (\varphi(v))$ for all $g\in G$ and all
  $v\in V$.
\end{definition}

\begin{definition}
  If $V$ and $W$ are representations of $G$, then the \emph{tensor product}
  $V\otimes W$ is a representation of $G$ under the action
$$g\cdot (v\otimes w) = (g\cdot v) \otimes (g\cdot w).$$
\end{definition}

\begin{definition}
  If $V$ is a representation of $G$, then $V^\vee=\Hom(V,\mathbb{C})$
  is a representation of $G$ under the action where, for $f\in V^\vee$
  and $g\in G$, $g\cdot f$ is the functional defined by
$$(g\cdot f)(v) = f(g^{-1}\cdot v)$$
for any $v\in V$.
\end{definition}

We need two definitions specific to the group $S_n$.

\begin{definition}
  For any $n$, the representation $\varepsilon$ is the one dimensional
  vector space on which $S_n$ acts by sign.  This means, for
  $v\in\varepsilon$, $g\cdot v=v$ if $g$ is an even permutation and
  $g\cdot v= -v$ if $g$ is an odd permutation.

  If $V$ is any representation of $S_n$, then $V\otimes\varepsilon$ is
  isomorphic to $V$ as a vector space, but the action of $S_n$ differs
  in that the action of an odd permutation picks up a sign.
\end{definition}

\begin{definition}
  Given a word $w$, the representation $V(w)$ is the vector space with
  basis given by the rearrangements of $w$, with $S_n$ acting by
  permuting our basis according to how it rearranges words.
\end{definition}

The representation $V(w)$ is special in that the action of $S_n$ is
actually induced from a combinatorial action of $S_n$ on a basis of
$V(w)$.  This property has a useful consequence.

\begin{lemma}
  Given any word $w$ of length $n$, we have a canonical isomorphism of
  $S_n$ representations $V(w)\simeq V(w)^\vee$ given by identifying our
  basis of rearrangements with its dual basis.
\end{lemma}

\begin{proof}
Since $V(w_1)$ has a canonical basis $\{v_r\}$, where $r$ is an arbitrary rearrangement, $V(w_1)^\vee$ has a dual basis $\{f_r\}$, and
$$
(\sigma\cdot f_r)(v_{r^\prime}) =  \left\{
     \begin{array}{ll}
       1 & \mbox{  if } r^\prime = \sigma\cdot r \phantom{.}\\
       0 & \mbox{  if } r^\prime \neq \sigma\cdot r.
     \end{array}
   \right.
$$
Hence $\sigma\cdot f_r = f_{\sigma\cdot r}$, and the map sending $v_r$
to $f_r$ is an isomorphism of $S_n$ representations.
\qed
\end{proof}

We are now ready to construct representations of $S_n$ using the
combinatorics of complementary words discussed above (see Definition~\ref{def:complementary_rearrangements}).

\begin{corollary} \label{corollary:irrep} Suppose $w_1$ and $w_2$ are
  words of length $n$ with complementary rearrangements.  Then there
  is a unique-up-to-scaling map of $S_n$ representations
$$
\varphi \colon V(w_2) \otimes \varepsilon \longrightarrow V(w_1),
$$
and the image of $\varphi$ is an irreducible representation.
\end{corollary}

\begin{proof}
Consider an arbitrary map of $S_n$ representations (meaning a linear map $\Psi$ where $\sigma\cdot \Psi(v) = \Psi(\sigma\cdot v)$ for all $v$)
$$\Psi: V(w_2) \otimes V(w_1)^\vee \rightarrow \epsilon.$$
Any such map must factor through the quotient 
$$
V(w_2) \otimes V(w_1)^\vee / W,
$$
where $W$ is the subspace spanned by elements of the form $(\mathrm{sign}(\sigma)v - \sigma v)$, for any $v$.  In this quotient,
any element of $S_n$ acts on the image of any vector by sign.  Pairs of rearrangements form a basis for the tensor product, so the images of these basis vectors still span the quotient.  Suppose some pair of rearrangements has a repeated column; then swapping those columns fixes the pair.  The vectors indexed by such pairs become zero in the quotient because transpositions are odd.

By Theorem~\ref{theorem:diagrams}, the action on pairs has a unique free orbit.  Any two vectors in the free orbit are related by a unique permutation, and so any vector spans the quotient, which must therefore be one-dimensional.  Using tensor-hom adjunction,
\begin{align*}
\Hom(V(w_2) & \otimes V(w_1)^\vee/W\;, \; \varepsilon) \\ & \simeq \Hom(V(w_2) \otimes V(w_1)^\vee\;,\; \varepsilon) & 
\\  & \simeq \Hom(V(w_2)\otimes \varepsilon \;,\; V(w_1) ), &
\end{align*}
where the first space is one-dimensional by the previous argument.  (Note that all isomorphisms are natural.)  We may take $\varphi$ to be any nonzero vector in the last hom-space.  This shows we have a unique-up-to-scaling linear map
$$
\varphi \colon V(w_2) \otimes \varepsilon \longrightarrow V(w_1).
$$

Now we show $V = \im \varphi$ is irreducible. Suppose $U \subset V$ is a proper subrepresentation.  By Maschke's theorem, there exists a complementary subrepresentation $U' \subseteq V$ with the property that $V = U \oplus U' $.  Let $\pi \colon V \to V$ denote the projection with kernel $U'$ and image $U$.  But now the composite
$$
V(w_2) \otimes \varepsilon \longrightarrow V \overset{\pi}{\longrightarrow} V \longrightarrow V(w_1)
$$
cannot be a scalar multiple of $\varphi$ since it is nonzero and has a different image.
\qed
\end{proof}

\noindent
The following definition will help us write an explicit example of the linear map $\varphi$.

\begin{definition} \label{definition:YoungsChar}
Let $w_1, w_2$ be fixed words of length $n$, and let $r_1$ and $r_2$ be arbitrary rearrangements of $w_1$ and $w_2$ respectively.  Define \emph{Young's character} 
$$
\mathrm{Y}_{w_1,w_2}(r_1, r_2) = \sum_{\sigma} \mathrm{sign}(\sigma),
$$
where $\sigma \in S_n$ ranges over all permutations such that $\sigma \cdot w_1 = r_1$ and $\sigma \cdot w_2 = r_2$.
\end{definition}

\begin{proposition} \label{proposition:young}
Young's character takes values in $\{-1, 0, 1\}$.   Whenever writing $r_2$ on top of $r_1$ has a repeated column, we have $\mathrm{Y}_{w_1,w_2}(r_1, r_2) = 0$.  If $w_1$ and $w_2$ are complementary, then $\mathrm{Y}_{w_1,w_2}(r_1, r_2) \neq 0$ exactly when all columns are distinct.
\end{proposition}
\begin{proof}
If there is a repeated column, then flipping those columns does not change the value of $\mathrm{Y}$ (since the inputs are the same), but it also introduces a sign change (since flipping two columns is an odd permutation).  It follows that $\mathrm{Y} = 0$ in this case.  If all the columns are distinct, then there is at most one permutation carrying each row back to $w_i$, and so the sum either is empty or has a single term.  In the event that $w_1$ and $w_2$ are complementary, the sum is nonempty.
\qed
\end{proof}

\begin{definition}\label{SpechtMatrix}
If $w_1, w_2$ are complementary words of length $n$, the \emph{Specht matrix} $\varphi(w_1,w_2)$ is the 
$$
\{ \mbox{ rearrangements of $w_1$ } \} \times \{ \mbox{ rearrangements of $w_2$ } \}
$$
matrix with $(r_1, r_2)$-entry $\mathrm{Y}_{w_1,w_2}(r_1, r_2)$.  If $w_1$ and $w_2$ are not complementary but have complementary rearrangements, we choose complementary rearrangements $w_1^\prime$ and $w_2^\prime$ of $w_1$ and $w_2$ respectively and define the Specht matrix $\varphi(w_1,w_2)$ to be $\varphi(w_1^\prime,w_2^\prime)$.  The column-span of the Specht matrix (as a subspace of $V(w_1)$) is the \emph{Specht module} $V(w_1,w_2)$.
\end{definition}

Note that, in the case where $w_1$ and $w_2$ are not themselves complementary, the Specht matrix $\varphi(w_1,w_2)$ is only defined up to a global choice of sign (depending on which complementary rearrangements are chosen), but the Specht module is the same regardless of this choice.

The symmetric group acts on $V(w_1, w_2)$ by permuting the rearrangements of $w_1$.

\begin{example}\label{action_on_(2,2)}
Let $w_1=1122$ and $w_2=1212$.  Then the Specht matrix $\varphi(1122, 1212)$ is shown in Table~\ref{table:specht-matrix22}.

\begin{table}[!htb]
\normalsize
\begin{center}
$\begin{array}{c|cccccc} \label{SpechtMatrix2+2_1}
        & 1122 & 1212 & 1221 & 2112 & 2211 & 2121 \\ \hline
 1212  & 1 & 0 & -1 & -1 & 1 & 0 \\
 1122  & 0 & -1 & 1 & 1 & 0 & -1 \\
 1221  & -1 & 1 & 0 & 0 & -1 & 1 \\
 2121  & 1 & 0 & -1 & -1 & 1 & 0 \\
 2211  & 0 & -1 & 1 & 1 & 0 & -1 \\
 2112  & -1 & 1 & 0 & 0 & -1 & 1. \\
\end{array}$
\caption{Specht matrix $\varphi(1122,1212)$}
\label{table:specht-matrix22}
\end{center}
\end{table}
We will describe now the action of $S_n$ on the rows of this Specht matrix. Consider the action of
$\sigma = (123)$ on the row word $w_1$. This action changes the order of rows to $2112$, $1212$, $2211$, $1221$, $2121$ and $1122$. What effect does it have on the columns of the Specht matrix? Let us look, for example, at the first column, labelled by $1122$. With the new row order this column becomes $c = (-1,1,0,-1,1,0)$, which is the original column for $r_2=2112$. If we consider also the action of $\sigma$ on the rearrangements of $w_2$, we see that the rearrangement $r_2=2112$ of $w_2$ becomes $\sigma\cdot2112 = 1122$, which is the label of the first column. The permutation $\sigma$ acts this way on the Specht module. 
\end{example}

We have the following fundamental fact about this representation.

\begin{theorem} \label{theorem:SpechtIrred}
The Specht module $V(w_1,w_2)$ is an irreducible representation of $S_n$.
\end{theorem}
\begin{proof}
By the proof of Corollary~\ref{corollary:irrep}, we see that having a unique free orbit gives a unique-up-to-scaling map $\varphi \colon V(w_2)\otimes \varepsilon \to V(w_1)$ whose image is irreducible.  It remains only to show that the Specht matrix provides an explicit choice for $\varphi$.  This was accomplished in Proposition~\ref{proposition:young}, which shows that $\mathrm{Y}$ provides a map 
$$
\varepsilon \to V(w_1) \otimes V(w_2)^\vee,
$$
where we have used the fact that the action of $S_n$ on $V(w_2)$ and $V(w_2)^\vee$ are canonically equal.
\qed
\end{proof}

The representation does not actually depend on the words but only on the partition diagram, so we make the following definition:
\begin{definition}\label{SpechtModule}
If $\lambda$ is a partition of $n$, then the \emph{Specht module} $V(\lambda)$ is the Specht module $V(w_1,w_2)$ for any choice of complementary $w_1$ and $w_2$ such that the diagram showing $w_1$ and $w_2$ are complementary is $D(\lambda)$.  We call $w_1$ the \emph{row word} and $w_2$ the \emph{column word}.
\end{definition}

For example, the matrix in Example~\ref{action_on_(2,2)} is the Specht matrix $V(2,2)$.

\begin{remark}
  Since every entry of the Specht matrix is a $0$, $1$, or $-1$, the Specht
  module can be similarly defined over any field.  However, over a
  field of positive characteristic, Maschke's Theorem does not hold.
  Nevertheless, over any field of characteristic other than $2$, the
  statements above show that the Specht module is \emph{indecomposable}, meaning that $V(w_1,w_2)$ cannot be written as
  the direct sum of two subrepresentations.
\end{remark}

Note that, if $w_1$ and $w_2$ are complementary with associated
diagram $D$, then $w_2$ and $w_1$ are also complementary, with an
associated diagram which is the transpose of $D$.  It will be useful
to have a definition describing this relationship.

\begin{definition}
Given a partition $\lambda$, the \emph{conjugate partition} $\lambda^*$ is the partition whose diagram $D(\lambda^*)$ is the transpose of the diagram $D(\lambda)$.  Formally, we have
$$\lambda^*_i = \#\{k\mid \lambda_k\geq i\}.$$
\end{definition}
\noindent
For example, if $\lambda = (4, 3, 1, 1)$, then $\lambda^*= (4, 2, 2, 1)$.  This natural combinatorial construction has representation-theoretic meaning, as the next result shows.

\begin{theorem} \label{theorem:transposes} We have an isomorphism of
  $S_n$ representations $V(\lambda^*)\simeq V(\lambda)^{\vee} \otimes
  \varepsilon$.
\end{theorem}
\begin{proof}
Observe that transposing the Specht matrix $\phi(w_1, w_2)$ gives the Specht matrix $\phi(w_2, w_1)$, which is the Specht matrix for the conjugate partition.  After transposing, however, the symmetric group acts on $\phi(w_2, w_1)$ by rearranging the \emph{column word} $w_1$.  This is not the correct action of the symmetric group on the Specht matrix.  However, the action is off only by a sign and a dual because, by the identity
$$
\mathrm{Y}(w_2, \sigma w_1) = (-1)^\sigma \cdot \mathrm{Y}(\sigma^{-1} w_2, w_1),
$$
the symmetric group acts on the row word $w_2$ by $\sigma^{-1}$ and picks up a sign with odd permutations.  
\qed
\end{proof}
\begin{remark}
We have not needed it, but it is actually the case that Specht modules are self-dual in the sense that there is an abstract isomorphism $V(\lambda)^{\vee} \simeq V(\lambda)$.  Choosing a basis from the columns of the Specht matrix, it would be possible to write matrices for the action of the symmetric group.  Evidently, these matrices would contain only real numbers---in fact, only rational numbers---and so their traces would be real as well.  It follows that the character of a Specht module is real, and so its dual, whose character is given by complex conjugation, is the same.  With this fact in mind, Theorem~\ref{theorem:transposes} gives that $V(\lambda^*)\cong V(\lambda) \otimes \varepsilon$.
\end{remark}

We remark briefly on the relationship between the construction above and the more usual presentation found, for example, in James and Kerber~\cite[Chapter 7.1]{JK} or Sagan~\cite[Chapter 2.3]{Sagan}.  In the usual construction, one defines a \emph{column-strict filling} of $\lambda$ to be a labelling of $D(\lambda)$ by the integers $\{1,\ldots,n\}$ such that every column is strictly increasing.  Then, for each column-strict filling $T$, one associates an element $v_T$ in an abstractly defined vector space, and the Specht module is the span of the vectors $v_T$ as one takes all possible fillings $T$.  In the definition of Specht module $V(\lambda)$ used here (Definition~\ref{SpechtModule}), we start with a word $w_2$, which we can take to be $w_2=1^{\mu_1}2^{\mu_2}\cdots k^{\mu_k}$, where $\mu=\lambda^*$ and $k=\lambda_1$.  Then each rearrangement $r$ of the word $w_2$ gives a column of the corresponding Specht matrix, which we can interpret as a vector $v_r$, and $V(\lambda)$ is defined as the span of the vectors $v_r$ as one takes all possible rearrangements $r$.  For each rearrangement $r$ of $w_2$, one can define an associated filling $T_r$: the one where the labels in column $i$ are the positions of the appearances of $i$ in $r$. For example, if $\lambda=(4,3,1,1)$, so $w_2={111122334}$, and we take $r = {131243112}$ (or $r=\mbox{ESENTSEEN}$ if we let $w_2=\mbox{TENNESSEE}$), then $T_r$ is
$$
T_r = \begin{ytableau}
1 & 4 & 2 & 5 \cr
3 & 9 & 6 \cr
7 \cr
8
\end{ytableau}.
$$

This correspondence between column-strict fillings and rearrangements essentially gives the correspondence between our version and the usual version.  Our version, however, sometimes differs by a sign (when the minimal rearrangement of $w_2$ to $r$ is by an odd permutation).  This sign turns out to be useful; for example, Theorem~\ref{theorem:allvertices} would be harder to state otherwise.

\section{The Specht Modules are a Complete Set of Irreducible Representations}

The Specht modules $V(\lambda)$ as $\lambda$ varies over all partitions of $n$ form a complete set of finite-dimensional irreducible representations for $S_n$.  The usual proof, found for example in~\cite[Section 2.4]{Sagan}, shows that $V(\lambda)\not\cong V(\mu)$ for $\lambda\neq\mu$; hence, since there are as many conjugacy classes of $S_n$ as there are partitions of $n$, we have all the irreducible representations.  We give a new and different argument here that the Specht modules are a complete set of irreducibles assuming an unproven combinatorial conjecture.

\begin{definition}\label{Definition:hook}
In a diagram $D$, the \emph{hook} of a box $d \in D$ consists of the box $d$, all boxes directly to the right of $d$, and all boxes directly below $d$.  In other words, if $d=(i,j)\in D$, the hook of $d$ is the set of all $(a,b)\in D$ such that $a\geq i$ and $b=j$ or $a=i$ and $b\geq j$.  The \emph{hook length} of a box $\Gamma(d)$ is the number of boxes in its hook.
\end{definition}
\begin{figure}[!ht] 
\centering
\includegraphics{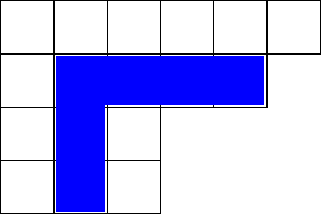}
\caption{A hook with $6$ boxes inside the diagram $D(\lambda)$ for $\lambda = (6,5,3,3)$}
\label{figure:hook}
\end{figure}
\begin{definition}
The \emph{dimension} of a diagram $D$ with $n$ boxes is defined as
$$
\dim D = \frac{n!}{\prod_{d \in D} \Gamma(d)}.
$$
\end{definition}
By a beautiful result of Frame, Robinson, and Thrall~\cite{FRT}, the dimension of a diagram equals the number of standard Young tableaux, which is the dimension of its Specht module.  A bijective proofs of this result was later given by Novelli, Pak, and Stoyanovskii~\cite{NPS}.  Consequently, $\dim D$ is always an integer.

An \emph{ordered set partition} of a finite set $S$ is a sequence $(P_1,\ldots,P_\ell)$ of subsets of $S$ such that each $P_i$ is nonempty, $P_i\cap P_j=\emptyset$ for all $i\neq j$, and $\bigcup_i P_i = S$.  An ordered set partition is \emph{properly ordered} if the parts are nonincreasing in size, so $\# P_i \geq \# P_{i+1}$ for all $i$.  For each properly ordered set partition $P=(P_1,\ldots,P_\ell)$ of the set $\{1, \ldots, n\}$, pick a word $w_P$ of length $n$ so that the $i$-th and $j$-th letters of $w_P$ match if and only if $i$ and $j$ are in the same $P_k$.  For each $P$, choose also a word $w^\prime_P$ so that $w_P$ and $w^\prime_P$ are complementary.  Every properly ordered set partition $P$ gives rise to an underlying partition $\lambda(P)$ with $\lambda(P)_i=\#P_i$.  Note that a set partition with parts of distinct sizes will have only one proper ordering, but a set partition with parts of equal sizes will have more.  Let $d(P)=\dim D(\lambda(P))$.

While searching for a Specht matrix proof that every irreducible representation of the symmetric group is isomorphic to some Specht module, we were led to the following conjecture, which has been checked for $n \leq 5$.

\begin{conjecture} \label{conjecture:funnysum}
If $\sigma, \tau \in S_n$ are two permutations, then Young's character in Definition~\ref{definition:YoungsChar} satisfies
$$
\sum_{P} \sum_{r} d(P)^2 \mathrm{Y}(\sigma w_P, r) \mathrm{Y}(\tau w_P, r) =  \left\{
     \begin{array}{lr}
       (n!)^2 & \mbox{  if } \sigma = \tau \phantom{,}\\
       0 & \mbox{  if } \sigma \neq \tau,
     \end{array}
   \right.
$$
where the first sum is over all properly ordered set partitions of $\{1, \ldots, n\}$, and the second sum is over all rearrangements of the word $w^\prime_P$.
\end{conjecture}

Note that, for our application, we desire a proof of the conjecture that does {\em not} make use of the following theorem.

\begin{theorem} \label{theorem:complete}
Every irreducible representation of $S_n$ arises exactly once from a diagram with $n$ boxes.  
\end{theorem}
\begin{proof}
This proof assumes Conjecture~\ref{conjecture:funnysum}.  There are three parts: first, we build a block matrix whose blocks are built from Specht matrices; second, we use the conjecture to show that this matrix has full rank; finally, we conclude that the regular representation $\mathbb{C}S_n$ is spanned by a sum of Specht modules.

Build a block matrix $M$ with a single block row and a block column for every properly ordered set partition $P$ of the set $\{1, \ldots, n\}$.  The block $M_P$ associated to $P$ is a matrix whose rows are indexed by $S_n$ and whose columns are indexed by the rearrangements of $w^\prime_P$.  The $(\sigma, r)$-entry of $M_P$ will be $\mathrm{Y}_{w_P,w^\prime_P}(\sigma w_P, r)$.  The rows of $M_P$ come directly from the Specht matrix associated to the complementary pair $(w_P, w^\prime_P)$, so the column-span of $M_P$ is isomorphic to the Specht module $V(\lambda_P)$.

Conjecture~\ref{conjecture:funnysum} asserts that the rows of this block matrix are orthogonal under the inner product given by the diagonal inner product $\langle u,v\rangle=\sum_{P;r} d(P)^2 \cdot (u_{P;r}v_{P;r})$, where $u_{P;r}$ denotes the entry of $u$ in the column indexed by $P$ and $r$.  Consequently, the block matrix $M$ has full rank.

The natural action of $S_n$ permuting the rows of this matrix is the regular representation $\mathbb{C}S_n$.  Since $M$ has full rank, the image of $M$ must be the regular representation.  On the other hand, the image of $M$ is the span of the images of $M_P$, and the image of each $M_P$ is isomorphic to some Specht module $V(\lambda_P)$.  Hence the regular representation is spanned by a sum of Specht modules.

The regular representation always contains a copy of every irreducible representation, so every
irreducible representation of $S_n$ must be isomorphic to the Specht module $V(\lambda)$ for some partition $\lambda$.  Since there are as many conjugacy classes of $S_n$ as there are partitions of
$n$, and the number of distinct irreducible representations is always equal to the number of conjugacy classes, the Specht modules must in fact be distinct.
\qed
\end{proof}

\section{Specht Matroids} \label{section:spechtmatroids}

A \emph{matroid} is a combinatorial encoding of the dependence relations among a finite set of vectors in a vector space.  This encoding can be defined in many equivalent ways, each with an axiomitization on some collection of subsets of a \emph{ground set} $E$, which is the abstraction of our original set of vectors.  Because of the presence of many equivalent definitions, each useful in a different context, it has become customary not to define the word ``matroid'' but instead to only give definitions of the \emph{bases}, \emph{independent sets}, \emph{circuits}, \emph{rank function}, or other linear algebra notion associated to a nebulous underlying object $M$, the ``matroid''.

Since we will only work with matroids that actually come from a set of vectors in a $\mathbb{C}$-vector space (called \emph{$\mathbb{C}$-representable matroids}), we will not give any of these abstract definitions. We refer the interested reader to \cite{Oxley}.

Let $E$ be a finite set, which we take to be $\{1,\ldots,k\}$ for convenience, and let $\{v_i\mid i\in E\}$ be a set of vectors spanning a vector space $\mathbb{C}^n$.  Then a subset $B\subseteq E$ is a \emph{basis} of the matroid $M(v_1,\ldots,v_k)$ if $\{v_i\mid i\in I\}$ is a basis of $\mathbb{C}^n$.  A subset $C\subseteq E$ is a \emph{circuit} of $M$ if it is a \emph{minimal dependent set}, meaning that $v(C) = \{v_i\mid i\in C\}$ is dependent but any proper subset of $v(C)$ is independent.  Given some subset $A\subset E$ the \emph{rank} of $A$, denoted $r(A)$, is the dimension of the subspace spanned by $\{v_i\mid i\in A\}$.  A \emph{flat} of $M$ is a maximal subset of $E$ of a given rank; in other words, $F\subseteq E$ is a flat if $r(F\cup\{i\})>r(F)$ for all $i\not\in F$.  One can think of each flat $F$ as representing the subspace spanned by $\{v_i\mid i\in F\}$; this gives a one-to-one correspondence between subspaces spanned by a subset of $\{v_1,\ldots,v_k\}$ and flats.  This correspondence shows that the set of flats of a matroid $M$ forms a lattice under inclusion; this is the \emph{lattice of flats} of $M$.

Given a partition $\lambda$, we define the \emph{Specht matroid} $M(\lambda)$ to be the matroid formed from the columns of the Specht matrix for $\lambda$. Note the Specht matrix depends on a global choice of sign coming from the complementary words chosen, but the matroid is independent of this choice.

\begin{example}\label{matroid_on_(2,2)}
We describe the matroid $M(2,2)$, which is the matroid represented by the columns of the Specht matrix in Example~\ref{action_on_(2,2)}.  The circuits are $\{1122, 2211\}$, $\{1212, 2121\}$, and $\{1221, 2112\}$ and all sets of 3 vectors not containing one of the first 3 sets.  The bases are the $12$ different sets of 2 vectors not containing a circuit.  (One can choose any of the 6 vectors as the first vector in the basis, which leaves 4 choices for the second vector, but this procedure chooses every basis twice.)  The 5 flats are $\emptyset$, the 3 circuits of size 2, and the set of all 6 vectors.
\end{example}

We can characterize the possible circuits of size 2, which also characterizes the flats of rank 1.

\begin{theorem}
The Specht matroid $M(\lambda)$ has a circuit with two elements if and only if the diagram of $\lambda$ has two columns of the same size.
\end{theorem}
\begin{proof}
For simplicity, we let $\mu=\lambda^*$ and let our complementary words be $x$ and $w$, where $w=1^{\mu_1}\cdots k^{\mu_k}$ (with $k=\lambda_1$) and $x=12\cdots\mu_11\cdots\mu_2\cdots1\cdots\mu_k$.  Note that $(x,w)$ is indeed part of the free orbit on rearrangements of $(x,w)$.  We could give a proof describing the circuits involving  any arbitrary rearrangement $r$ of $w$, but  since the symmetric group $S_n$ acts transitively on the matroid $M(\lambda)$, we can choose our favorite rearrangement of $w$, which will be $w$ itself, and for any other rearrangement $r$, we will have a circuit of two elements involving $r$ if and only if there is a circuit of two elements involving $w$.

Suppose $\lambda$ has two columns of the same size, or, in other words, there exists some $i$ such that $\mu_i=\mu_{i+1}$.  Consider $w$ and the rearrangement $r$ where all the $i$'s and $(i+1)$'s have been switched, so $r=1^{\mu_1}\cdots(i-1)^{\mu_{i-1}}(i+1)^{\mu_i}(i)^{\mu_i}(i+2)^{\mu_{i+2}}\cdots k^{\mu_k}$.  For any rearrangement $s$ of the row word, we have $Y(s,w)=(-1)^{\mu_i} Y(s, r)$.  Hence, $v_w- (-1)^{\mu_i} v_r = 0$, and $\{w, r\}$ forms a circuit.

Now, suppose all the columns of $\lambda$ are distinct.  Suppose $r$ is some rearrangement of $w$ (with $r\neq w$).  We show there is some rearrangement $s$ of the row word such that $Y(s,w)\neq Y(s,r)$.  Let $k$ be the smallest integer such that $w_k\neq r_k$, and let $i=w_k$ and $j=r_k$.  By our construction of $w$, $j>i$, and since the columns of $\lambda$ are distinct, $\mu_j<\mu_i$.  Also, $r_M=i$ for some $M>m$, where $m=\mu_1+\cdots+\mu_i$ is the position of the last occurrence of the letter $i$ in $w$.  Now consider the rearrangement $s$ of the row word switching $x_k$ and $x_m$. Note that $(s,w)$ is part of the free orbit, and, in fact, $Y(s,w)=-1$.  However, $Y(s,r)=0$ since $(s_M,r_M)=(a,i)$ for some $a<x_k$ and also $(s_{k-x_k-a},r_{k-x_k-a})=(a,i)$.  Hence $v_r$ and $v_w$ do not form a circuit for any rearrangement $r$.
\qed
\end{proof}

Matroids have a number of interesting invariants.  One is the characteristic polynomial, a generalization of the chromatic polynomial for a graphical matroid.  The characteristic
polynomial $p_M(t)$ for a matroid $M$ can be calculated recursively by deletion and contraction, so it is a specialization of the Tutte polynomial $T_M(x,y)$.  It would be interesting to find formulas
or characterizations of the Tutte or characteristic polynomials of Specht matroids.
Another important invariant is the Chow ring of a matroid, which we address in the following section.

\begin{example}
For the Specht matroid $M(2,1,1,1)$, we use \emph{Sage} to compute the Tutte polynomial, which is
$$T_{M(2,1,1,1)}(x,y) = x^4 + x^3 + x^2 + x + y.$$
One obtains the characteristic polynomial from the Tutte polynomial by the formula $p_M(t) = (-1)^{r(M)}T_M(1-t,0)$. 
Since $r(M(2,1,1,1)) = 4$, we get $$p_{M(2,1,1,1)}(t)= t^4 - 5t^3 + 10t^2 - 10t + 4.$$

Similar computations produce
$$p_{M(3,2)}(t)=t^5 - 15t^4 + 90t^3 - 260t^2 + 350t - 166$$
and
$$p_{M(2,2,1)}(t) = t^5 - 10t^4 + 45 t^3 - 105 t^2 + 120t - 51.$$
\end{example}

\section{Chow Rings} \label{section:chow}

Given a matroid $M$, Feichtner and Yuzvinksy~\cite{FY} (following DeConcini and Procesi~\cite{DP} in the representable case) defined the \emph{Chow ring} of $M$, which we denote by $A^*(M)$, to be the ring $\mathbb{Q}[x_F]/I_M$ presented as follows.  There is one generator $x_F$ for each nonempty flat $F$, and the ideal $I_M$ is generated by the following two types of relations:
\begin{itemize}
\item $x_F x_G \in I_M$ whenever $F$ and $G$ are incomparable, and.
\item $\sum_{F\supset \{e\}} x_F \in I_M$ for every element $e$ in the ground set.
\end{itemize}

The definition of Feichtner and Yuzvinsky also requires as input a
\emph{building set}, which is a subset of the flats satisfying some
combinatorial properties with respect to the lattice.  The definition
we have given here is the case of the {\it maximal} building set,
which is the one containing every nonempty flat.

Note that a slightly different presentation of the Chow ring appears also in the
literature. In~\cite{AHK}, Adiprasito, Huh, and Katz use a presentation that differs
from the one of Feichtner and Yuzvinksy in not
using the generator (which can be rewritten in terms of other
generators) corresponding to the entire ground set of $M$.

The next example gives a solution to Problem~1 on Grassmannians in \cite{Sturmfels}.

\begin{example}\label{Example:apprenticeship1}
Let us consider the matroid $M$ which is formed from the columns of the matrix 
$$
\left(
\begin{array}{cccccc}
 1 & 0 & 0 & 1 & 1 & 1 \\
 0 & 1 & 0 & 2 & 3 & 4 \\
 0 & 0 & 1 & 0 & 0 & 1 \\
\end{array}
\right),
$$
where the columns are labelled $0,\ldots,5$.  This matrix represents a point in $\Gr(3,\mathbb{C}^6)$ with $16$ nonzero Pl\"ucker coordinates.
 
To determine the Chow ring of $M$, first we list all nonempty flats of $M$:
\begin{multline*}
\{0\}, \{0, 1, 2, 3, 4, 5\}, \{0, 1, 3, 4\}, \{0, 2\}, \{0,
5\}, \{1\}, \{1, 2\}, \{1, 5\}, \\
\{2\}, \{2, 3\}, \{2,
4\}, \{2, 5\}, \{3\}, \{3, 5\}, \{4\}, \{4, 5\},
\{5\}.
\end{multline*}
There is one generator $x_F$ of the Chow ring $A^*(M)$ for each
nonempty flat $F$.

The monomial generators in $I_M$ coming from pairs of incomparable
flats are
\begin{multline*}
x_{0}x_{12}, x_{0}x_{15}, x_{0}x_{23}, x_{0}x_{24},
x_{0}x_{25}, x_{0}x_{35}, x_{0}x_{45}, \\ x_{1}x_{02},
x_{1}x_{05}, x_{1}x_{23}, x_{1}x_{24}, x_{1}x_{25},
x_{1}x_{35}, x_{1}x_{45}, \\ x_{2}x_{05}, x_{2}x_{15},
x_{2}x_{35}, x_{2}x_{45}, x_{2}x_{0134}, \\ x_{3}x_{02},
x_{3}x_{12}, x_{3}x_{05}, x_{3}x_{15}, x_{3}x_{24},
x_{3}x_{25}, x_{3}x_{45}, \\ x_{4}x_{02}, x_{4}x_{12},
x_{4}x_{05}, x_{4}x_{15}, x_{4}x_{23}, x_{4}x_{25},
x_{4}x_{35}, \\ x_{5}x_{02}, x_{5}x_{12}, x_{5}x_{23},
x_{5}x_{24}, x_{5}x_{0134}.
\end{multline*}
The relations in $I_M$ of the second type are
\begin{multline*}
x_{0}+x_{02}+x_{05}+x_{0134}+x_{012345},\ 
x_{1}+x_{12}+x_{15}+x_{0134}+x_{012345}, \\
x_{2}+x_{02}+x_{12}+x_{23}+x_{24}+x_{25}+x_{012345},\ 
x_{3}+x_{23}+x_{35}+x_{0134}+x_{012345}, \\
x_4+x_{24}+x_{45}+x_{0134}+x_{012345},\ 
x_5+x_{05}+x_{15}+x_{25}+x_{35}+x_{45}+x_{012345}.
\end{multline*}
Copying these generators and relations into \emph{Macaulay 2} (either by hand
or using the \emph{Sage} code in Section~\ref{section:ChowCode}), we obtain
that the Hilbert series of $A^*(M)$ is $1+11T+T^2$.
\end{example}

We list the dimensions of $A^i(M(\lambda))$ for all partitions $\lambda$ of $n=4$ and $n=5$ in
Table~\ref{table:chow-dims}.

\begin{table}[!htb]
\normalsize
\begin{center}
$\begin{array}{c|rrrrrr}
\lambda & 0 & 1 & 2 & 3 & 4 & 5 \\
\hline
(4) & 1 & &&&& \\
(3,1) & 1 & 8 & 1 & & & \\
(2,2) & 1 & 1 & & & &  \\
(2,1,1) & 1 & 7 & 1 & & & \\
(1,1,1,1) & 1 & &&&& \\
(5) & 1 & &&&& \\
(4,1) & 1 & 41 & 41 & 1 & &\\
(3,1,1) & 1 & 303 & 2553 & 2553 & 303 & \ \ \ 1 \ \\
(2,2,1) & 1 & 151 & 541 & 151 & 1 & \\
(3,2) & 1 & 256 & 1026 & 256 & 1 & \\
(2,1,1,1) & 1 & 21 & 21 & 1 & & \\
(1,1,1,1,1) & 1 & &&&&
\end{array}$
\caption{Dimensions of Chow groups of $M(\lambda)$}
\label{table:chow-dims}
\end{center}
\end{table}

Note that every row of Table~\ref{table:chow-dims} is palindromic, meaning that $\dim
A^i(M(\lambda))= \dim A^{d-1-i}(M(\lambda))$ for all $i$, where $d=\dim
V(\lambda)$.  In fact, the Chow ring $A^*(M)$ satisfies an algebraic version of Poincar\'e
duality for any matroid $M$.  Feichtner and Yuzvinsky~\cite[Corollary
2]{FY} prove this fact for a representable matroid $M$ by showing that $A^*(M)$
is the cohomology ring of a smooth, proper algebraic variety.  This equality of dimensions
was recently extended to non-representable matroids by Adiprasito,
Huh, and Katz~\cite[Theorem 6.19]{AHK} as the first major step in
their proof of the log-concavity of coefficients of the characteristic
polynomial of an arbitrary matroid. Finding a combinatorial interpretation of these dimensions
remains an interesting open problem.

By finding a Gr\"obner basis for $I_M$ and determining the standard
monomials, Feichtner and Yuzvinsky describe a monomial basis for $A^*(M)$
as follows~\cite[Corollary 1]{FY}.

\begin{theorem}[Paraphrased from \cite{FY}, Corollary 1]\label{Theorem:basis_for_chow_ring}
  The ring $A^*(M)$ has a basis consisting of the monomial $1$ and monomials
  $\prod_{i=1}^k x_{F_i}^{d_i}$ such that $F_1\subseteq \cdots \subseteq F_k$
  and $0<d_i<\mathrm{rank}(F_i)-\mathrm{rank}(F_{i-1})$ for all $i$ (considering
  $\mathrm{rank}(F_0)=0$ by convention).
\end{theorem}

The next example illustrates how to use Theorem~\ref{Theorem:basis_for_chow_ring}.
\begin{example}\label{Example:apprenticeship2}
Consider the Specht matroid $M(2,1,1,1)$. This is a uniform matroid on $5$ elements. Let us denote the elements of its ground set by $\{0,1,2,3,4\}$. Using \emph{Sage} we get the list of nonempty flats of $M(2,1,1,1)$ in Table~\ref{table:2111-flats}. 

\begin{table}[!htb]
\normalsize
\begin{center}
$\begin{array}{c|c}
\text{rank} & \text{flats}  \\
\hline
1 & \{0\},\{1\},\{2\},\{3\},\{4\} \\
2 & \{0, 1\}, \{0, 2\}, \{0, 3\}, \{0, 4\}, \{1, 2\}, \{1, 3\}, \{1, 4\}, \{2, 3\}, \{2, 4\}, \{3, 4\}\\
3 &

\begin{tabular}[c]{@{}l@{}}\{0, 1, 2\}, \{0, 1, 3\}, \{0, 1, 4\}, \{0, 2, 3\}, \{0, 2, 4\}, \\ \{0, 3, 4\}, \{1, 2, 3\}, \{1, 2, 4\}, \{1, 3, 4\}, \{2, 3, 4\} \end{tabular}\\
4 & \{0,1,2,3,4\} 
\end{array}$
\caption{Nonempty flats of $M(2,1,1,1)$}
\label{table:2111-flats}
\end{center}
\end{table}
By considering one element sequences of flats we get one monomial of degree $1$ from each flat of rank greater than $1$. This gives $21$ monomials of degree $1$. Flats of rank $1$ do not contribute to the list of monomials since there is no integer $d$ such that $0<d<1$.

We can get a monomial of degree $2$ in two ways. Quadratic monomials which are a square of only one variable are obtained from one element sequences consisting of flats of rank greater than $2$. There are $11$ such flats. For the other quadratic monomials, we need to consider sequences of flats of length $2$ such that the rank of the first element and the difference between the ranks of the two elements are each at least $2$. In our case we get $10$ such sequences. They are of the form $F_k\subset F_{\{0,1,2,3,4\}}$, where $F_k$ is a flat of rank~$2$.

The only way to obtain a monomial of degree $3$ is from the one element sequence~$F_{\{0,1,2,3,4\}}$. Since the biggest rank of a flat in our matroid is $4$, there are no monomials of degree $4$ of higher. 

The dimensions we obtained are $1,21,21,1$. Note that this agrees with the next-to-last row of Table~\ref{table:chow-dims}.

\end{example}

In the case where a matroid $M$ is a Specht matroid, Theorem~\ref{Theorem:basis_for_chow_ring} has a rather appealing
consequence.  We say that a (finite dimensional) representation $V$ of
a (finite) group $G$ is a \emph{permutation representation} if the
group action on $V$ actually arises from an action of $G$ on a basis
of $V$.  In other words, this means $V$ has a basis $v_1,\ldots,v_d$
such that, for all $i$ with $1\leq i\leq d$ and all $g\in G$, $g\cdot
v_i = v_j$ for some basis element $v_j$.  These representations are
particularly easy to understand since one only has to understand the
combinatorics of a group acting on a finite set.

Given an element $v_r$ of (the ground set of) the Specht matroid, the
action of any permutation $\sigma\in S_n$ on $V(w_1)$ takes $v_r$ to
$(-1)^\sigma v_{\sigma^{-1}r}$.
Therefore, given a flat $F$, which we think of as the subspace spanned
by $\{v_{r_1},\ldots, v_{r_k}\}$, a permutation $\sigma$ sends $F$ to
the flat $\sigma^{-1}F$ corresponding to the subspace spanned by
$\{v_{\sigma^{-1}r_1},\ldots,v_{\sigma^{-1}r_k}\}$.  This action on flats induces an
action on the Chow ring $A^*(M(\lambda))$ sending a monomial
$\prod_{i=1}^k x_{F_i}^{d_i}$ to $\prod_{i=1}^k x_{\sigma^{-1}F_i}^{d_i}$.
Since a monomial satisfies the conditions of Theorem~\ref{Theorem:basis_for_chow_ring}
if and only if its image under the action of $\sigma$ does, the action of $S_n$
on $A^*(M(\lambda))$ is a permutation action.  Indeed, if one understood the action of $S_n$
on the set of flags of flats of $M(\lambda)$, one would be able to easily determine the graded
character of $A^*(M(\lambda))$ by substituting characters for dimensions in the computations
of Feichtner and Yuzvinsky~\cite[p. 526]{FY}

\section{Specht Polytopes} \label{section:spechtpolytopes}

Given a partition $\lambda$, we define the \emph{Specht polytope} $P(\lambda)$ to be the convex hull (in $\mathbb{R}^N$ where $N=\binom{n}{\lambda}$) of the columns of the Specht matrix.  Note this polytope is defined only up to a global sign; this choice will be irrelevant for our purposes since any polytope is projectively equivalent to its negative.

\begin{example}
Consider the partition $(2,1,1)$. The row and column words for this partition are $1123$ and $1211$, respectively. The Specht matrix is in Table~\ref{table:specht-matrix-211}.

\begin{table}[!htb]
\normalsize
\begin{center}
$
\begin{array}{c|cccc}
& 1211 & 1121 & 1112 & 2111 \\
\hline
1123 & 1 & 0 & 0 & -1 \\
1132 & -1 & 0 & 0 & 1 \\
1213 & 0 & -1 & 0 & 1 \\
1231 & 0 & 0 & 1 & -1 \\
1312 & 0 & 1 & 0 & -1 \\
1321 & 0 & 0 & -1 & 1 \\
2113 & -1 & 1 & 0 & 0 \\
2131 & 1 & 0 & -1 & 0 \\
2311 & 0 & -1 & 1 & 0 \\
3112 & 1 & -1 & 0 & 0 \\
3121 & -1 & 0 & 1 & 0 \\
3211 & 0 & 1 & -1 & 0 \\
\end{array}.$
\caption{Specht matrix for $(2,1,1)$}
\label{table:specht-matrix-211}
\end{center}
\end{table}
Using $\emph{Macaulay2}$ we can verify that the polytope in $\mathbb{R}^{12}$ which is the convex hull of the columns of the matrix in Table~\ref{table:specht-matrix-211} is a $3$-dimensional simplex. 
\end{example}

\begin{theorem} \label{theorem:allvertices}
Every column of the Specht matrix is a vertex of $P(\lambda)$.
\end{theorem}

\begin{proof}
Suppose some column of the Specht matrix could be written as a non-trivial convex combination of the others.  Since $S_n$ acts transitively on the columns of the Specht matrix, this would mean that every column could be written as a convex combination of the others, which would mean $P(\lambda)$ has no vertices, which is impossible.
\qed
\end{proof}

\begin{theorem}\label{Theorem:origin}
Every Specht polytope other than $P(1,\ldots,1)$ contains the origin.
\end{theorem}
\begin{proof}
First we will show that each row of a Specht matrix corresponding to a partition different from $(1,\ldots,1)$ contains the same number of $1$'s as of $-1$'s. Let $r$ be some permutation of a row word, and let $\mathrm{Stab}(r)$ be the set of all permutations preserving $r$. Since we excluded the partition $(1,\ldots,1)$, $\mathrm{Stab}(r)$ is a nontrivial direct product of symmetric groups. Nonzero entries in the row corresponding to $r$ in the Specht matrix have values $1$ or $-1$ depending on the signature of an element in $\mathrm{Stab}(r)$. Since $\mathrm{Stab}(r)$ has an equal number of odd and even permutations, the entries add up to $0$. So if $c_1,\ldots, c_m$ are the columns of the Specht matrix, then $\sum_{i=1}^m \frac{1}{m}c_i = 0$, which finishes the proof.


\qed
\end{proof}
\begin{theorem}\label{Theorem:dimSpechtPolytopeModule}
The dimension of the Specht polytope matches the dimension of the Specht module for any partition other than $(1,\ldots,1)$.
\end{theorem}
\begin{proof}
Since the Specht module is the span of the columns of a Specht matrix, its dimension is equal to the rank of this matrix. By Theorem~\ref{Theorem:origin}, the corresponding polytope contains the origin, so its dimension matches the dimension of the linear span of its vertices.
\qed
\end{proof}

We conclude this section with Table~\ref{table:f-vectors}, which gives $f$-vectors and dimensions of some of the Specht polytopes. We have not found yet any interpretation of these data. 

\begin{table}[!htb]
\normalsize
\begin{center}
$\begin{array}{c|c|c}
\lambda & \text{dimension} & f\text{-vector} \\
\hline
(3,1) & 3 & (1, 12, 24, 14, 1)\\
(2,2) & 2 & (1, 3,3,1)\\
(2,1,1) & 3 & (1, 4,6,4,1)\\
(4,1) & 4 & (1, 20, 60, 70, 30, 1)\\
(3,2) & 5 & (1, 15,60,80,45,12,1) \\
(3,1,1) & 6 & (1, 20, 120, 290, 310, 144, 24, 1)\\
(2,2,1) & 5 & (1, 10, 45, 90, 75, 22, 1)\\
(2,1,1,1) & 4 & (1, 5,10,10,5,1) 
\end{array}$
\caption{Dimensions and $f$-vectors of Specht polytopes}
\label{table:f-vectors}
\end{center}
\end{table}
\section{Examples: The Partitions $(2,1^{n-1})$ and $(n-1, 1)$} \label{section:examples}
In the previous section we saw that the Specht polytope for the partition $(2,1,1)$ is a simplex. In fact, any partition of the form $(2,1,\ldots,1)$ corresponds to a simplex.

\begin{theorem}
The Specht polytope $P(2,1,\ldots,1)$ is an $(n-1)$-dimensional simplex.
\end{theorem}
\begin{proof}

The Specht module $M(2,1,\ldots,1)$ of size $n$ has dimension $n-1$. By Theorem~\ref{Theorem:dimSpechtPolytopeModule} the Specht polytope has the same dimension. The partition
$(2,1\ldots,1)$ has the column word $121\cdots1$, which has $n$ rearrangements.  Hence the Specht polytope $P(2,1,\ldots,1)$ has at most $n$ vertices and dimension $n-1$, so it must be an $(n-1)$-simplex. 
\qed
\end{proof}

Correspondingly, for $\lambda=(2,1^{n-1})$, the matroid $M(\lambda)$ is the generic
matroid on $n$ elements of total rank $n-1$.  The Hilbert series for
$A^*(M)$ for any generic matroid $M$ was calculated by Feichtner and
Yuzvinsky.  It is not too difficult to modify their computation to
include information on the action of $S_n$. The dimensions of the Chow groups of this Specht matroid are in Table~\ref{table:21n-chow-dims}.

\begin{table}[!htb]
\normalsize
\begin{center}
$\begin{array}{c|rrrrr}
n \setminus k & 0 & 1 & 2 & 3 & 4 \\
\hline
2 & 1 & & & & \\
3 & 1 & 1 & & & \\
4 & 1 & 7 & 1 & & \\
5 & 1 & 21 & 21 & 1 & \\
6 & 1 & 51 & 161 & 51 & 1 \\
\end{array}.$
\caption{Dimensions of $A^k(M(2,1^{n-1}))$.}
\label{table:21n-chow-dims}
\end{center}
\end{table}
We make the following conjecture with the help of the \emph{OEIS} \cite{OEIS}.

\begin{conjecture}
The dimension of $A^k(M(2,1^{n-1}))$ is the number of permutations in $S_n$ with no fixed points and $k+1$ excedances (\emph{OEIS} A046739).
\end{conjecture}
A \emph{fixed point} of a permutation $\sigma\in S_n$ is an index $i$ such that $\sigma(i)=i$, and an \emph{excedance} is an index $i$
such that $\sigma(i)>i$.

Consider the cyclic subgroup $C_n\subset S_n$
generated by the $n$-cycle $c=(1 2 \cdots n)$.  Note that conjugation by $c$ preserves the number of fixed points and the
number of excedances, so, for any fixed $n$ and $k$, $C_n$ acts on the set of permutations of $S_n$ with no fixed points
and $k+1$ excedances.  Also, recall from the end of Section~\ref{section:chow} that $S_n$ acts on the Feichtner--Yuzvinsky
basis of $A^k(M(2,1^{n-1}))$, and $C_n$ acts on this basis by restricting
the action of $S_n$.  A possible refinement of the conjecture above is that the orbit structures of these two actions coincide; this refinement has been checked for $n\leq 6$.  For example, for $n=6$
and $k=2$, both actions have 1 orbit of size 1, 2 orbits of size 2, 4 orbits of size 3, and 24 orbits of size 6.

We switch our attention to  partitions of the form $(n-1,1)$. 
We start by describing the Specht matrices coming from this partition. 
\begin{proposition}
Each column of a Specht matrix for the partition $(n-1,1)$ is of the form $e_i-e_j$, where $e_i$ is a standard unit vector in $\mathbb{R}^n$, and, for $n\geq 4$, $e_i-e_j$ and $e_j-e_i$ are both columns of the Specht matrix for any $i,j$ with $1\leq i<j\leq n$.
\end{proposition}

\begin{proof}
A choice of row and column words for the partition $(n-1,1)$ is $w_1 = 11\cdots112$ and $w_2 = 123\cdots(n-1)1$, respectively.  Let $r_2$ be a rearrangement of the column word $w_2$. In the column for $r_2$, there are exactly two nonzero entries, namely the entries corresponding to
the rearrangements of $w_1$ in which the $2$ is in the same position as one of the $1$'s in $r_2$. Let us call these rearrangements $r_1$ and $r_1'$. If $\sigma$ is a permutation such that $\sigma w_1 = r_1$ and $\sigma w_2 = r_2$, and $\sigma'$ is a permutation such that $\sigma' w_1 = r'_1$ and $\sigma' w_2 = r_2$, then $\sigma$ and $\sigma'$ differ by a transposition, so we have $Y(r_1,r_2) = -Y(r'_1, r_2)$.

Given $i$ and $j$, to obtain $e_i-e_j$ and $e_j-e_i$, use the column corresponding to some rearrangement $r$ of $w_2$ with the $1$'s in the $i$-th and $j$-th positions and the column corresponding to rearrangement $r'$ obtained from switching the positions of the $2$ and $3$ in $r$.
\qed
\end{proof}

In what follows, we will denote a vertex of a Specht polytope $P(n-1,1)$ by $v_{ij}$ if the
corresponding column in the Specht matrix has $1$ in the $i$-th position and $-1$ in the
$j$-th position.
 
It turns out that the Specht polytopes $P(n-1,1)$ have already been
studied under the name of root polytopes by Ardila, Beck, Hosten,
Pfeifle, and Seashore~\cite{ABHPS}.

\begin{definition}
A \emph{root polytope} $P_{A_n}$ (of type $A_n$) is the convex hull of the points $e_i - e_j$ for $1 \leq i \neq j \leq n$ where $i, j \in \{1,\ldots,n\}$.  
\end{definition}

(Note this definition is different from the definition of Gelfand, Graev, and Postnikov~\cite{GGP}, which uses only the positive roots and zero.)

For this class of polytopes, Ardila, Beck, Hosten, Pfeifle and Seashore~\cite[Proposition 8]{ABHPS} gave the following description of their edges and facets. 

\begin{theorem} \label{theorem:root_polytope}
The polytope $P_{A_n}\in \mathbb{R}^{n+1}$ has dimension $n$ and is contained in the hyperplane $H_0 = \{x \in \mathbb{R}^{n+1} : \sum_{i=0}^n x_i = 0\}$. It has $(n-1)n(n+1)$ edges, which are of the form $v_{ij}v_{ik}$ and $v_{ik}v_{jk}$ for $i, j, k$ distinct. It has $2^{n+1}-2$ facets, which can be labelled by the proper subsets $S$ of $[0, n] := \{0, 1,\ldots, n\}$. The facet $F_S$ is defined by the hyperplane 
$$
H_S := \left\lbrace x\in \mathbb{R}^{n+1} : \sum_{i\in S} x_i= 1\right\rbrace,
$$
and it is congruent to the product of simplices $\Delta_S \times \Delta_T$ , where $T = [0, n] - S$.
\end{theorem}

The main idea in the proof of this theorem is that, if $f$ is a linear functional and $i,j,k,l$ are all different, then $f(v_{ij})+f(v_{kl})=f(v_{il})+f(v_{kj})$.  Hence $f$ cannot be maximized at only one of the line segments $v_{ij}v_{kl}$ or $v_{il}v_{kj}$, so neither can be an edge.  A similar argument works
both for ruling out pairs of vertices of the form $v_{ij}$ and $v_{jk}$ as edges and for determining the facets.

Ardila, Beck, Hosten, Pfeifle and Seashore also gave the following description of the lattice points inside $P_{A_n}$.  
\begin{theorem}
The only lattice points in $P_{A_n}$ are its vertices and the origin.
\end{theorem}
\begin{proof}
A polytope $P_{A_n}$ is contained in an $n-1$-sphere with radius $\sqrt{2}$ and center $0$. The only lattice points in this sphere are  $\pm e_i$ and $\pm e_i\pm e_j$ for $1\leq i,j \leq n$. Since $P_{A_n}$ is contained in the hyperplane $H_0 = \{x \in \mathbb{R}^{n+1} : \sum_{i=0}^n x_i = 0\}$, the only lattice points contained in $P_{A_n}$ are the vertices and the origin.    
\qed
\end{proof}

The matroid $M(n-1,1)$ is the matroid for the
braid arrangement of $S_n$.  This was one of the original motivations
of DeConcini and Procesi~\cite{DP} for studying Chow rings of representable matroids, since, in this case, $A^*(M)$ is
actually the cohomology ring for the moduli space
$\overline{\mathcal{M}_{0,n}}$ of $n$ marked points on the complex
projective line, which, as DeConcini and Procesi had earlier observed~\cite{DPWond}, can be realized as the successive blowup of
$\mathbb{P}^{n-2}$ at all the subspaces in the intersection lattice
of the braid arrangement.  For more information on $\overline{\mathcal{M}_{0,n}}$ and in
particular equations defining natural projective embeddings, see
the article in this volume by Monin and Rana~\cite{MR}.

Note that this matroid is also the graphical matroid on the complete graph
$K_n$ on $n$ vertices.  By the usual translation between graphical matroids and
graphs, vectors in the matroid correspond to (directed) edges of the graph, and a basis
of the matroid corresponds to a spanning tree for the graph.  To be precise,
we can label the vertices of $K_n$ by $\{1,\ldots,n\}$, and, with this labeling, the edge from $i$ to $j$
corresponds to the vector $e_j - e_i$.  If we start with the column word $w_2=1123\cdots n$, then
the vector $e_j - e_i$ is $v_r$ for the rearrangement
$r$ where a $1$ appears in the $i$-th and $j$-th positions and the remaining letters $2,\ldots, n$ appear
in order.  In terms of the usual presentation of the Specht matroid in terms of fillings, this vector corresponds to the filling
with an $i$ and a $j$ in the first column and the remaining integers in order along the first row.  The usual basis of the Specht
module given by Standard Young Tableaux corresponds to the tree with edges between
vertex $1$ and vertex $j$ for every $j>1$, and declaring $i$ to be the ``smallest''
letter in our filling alphabet gives a similar tree with vertex $i$ having degree $n-1$.  Of course, $K_n$
has many other types of spanning trees, so the Specht matroid has many bases that look completely different from this standard basis!

We will finish this section with a picture of one of these polytopes. For the partition $(3,1)$, the Specht polytope naturally lives in a four-dimensional space, but as it is a three-dimensional object, it can be drawn in a $3$-space. It is shown in Figure~\ref{figure:SpechtPolytope(3,1)}.

\begin{figure}
\centering
\includegraphics[scale = .8]{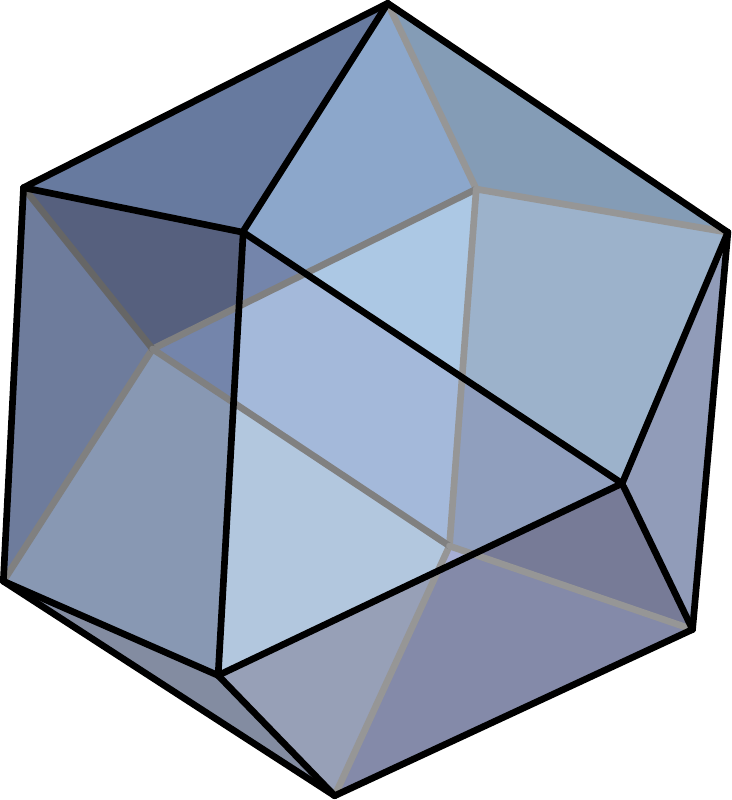}
\caption{Specht polytope $(3,1)$}
\label{figure:SpechtPolytope(3,1)}
\end{figure}

\section{Matroidification} \label{section:upgrading}

Many constructions in the representation theory of $S_n$ and their associated dimensions actually come from constructions involving tensor products or $\Hom$ spaces of Specht modules.  Since Specht modules have a distinguished symmetric spanning set, so do their tensor products, and hence we can extend our definitions of Specht matroids and Specht polytopes to these other contexts.
In this section, we build matroids and polytopes for three famous collections of numbers arising in combinatorics and representation theory: Kronecker coefficients, Littlewood-Richardson coefficients, and plethysm coefficients.
\begin{definition}
If $\lambda, \mu, \nu$ are partitions of $n$, the \emph{Kronecker coefficient} $g_{\lambda \mu\nu}$ is defined to be the dimension of the $S_n$-invariants of the tensor product
$$
g_{\lambda \mu\nu} = \dim \; (\mathrm{Specht}(\lambda) \otimes \mathrm{Specht}(\mu) \otimes \mathrm{Specht}(\nu))^{S_n}.
$$
\end{definition}

In Definitions~\ref{def:kron-matrix}, \ref{def:lr-matrix}, and\ref{def:pleth-matrix}, $x_1$ and $x_2$ 
denote a pair of complementary words of length $n$ that correspond via Theorem 
\ref{theorem:diagrams} to the partition $\lambda$.  Similarly, $y_1$ and $y_2$ correspond to $\mu$,
and $w_1$ and $w_2$ to $\nu$.

\begin{definition}
\label{def:kron-matrix}
The \emph{Kronecker matrix} has rows indexed by the product
$$
\{ \mbox{ rearrangements of $x_1$ } \} \times \{ \mbox{ rearrangements of $y_1$ } \} \times \{ \mbox{ rearrangements of $w_1$ } \} 
$$
and columns indexed by the product
$$
\{ \mbox{ rearrangements of $x_2$ } \} \times \{ \mbox{ rearrangements of $y_2$ } \} \times \{ \mbox{ rearrangements of $w_2$ } \}, 
$$
where the $((p, q, r), (s, t, u))$ entry is given by the formula 
$$
\sum_{\sigma \in S_n} \mathrm{Y}(\sigma s,p) \cdot \mathrm{Y}(\sigma t, q) \cdot \mathrm{Y}(\sigma u, r).
$$
Its columns define the \emph{Kronecker matroid}.
The convex hull of its columns defines the \emph{Kronecker polytope}.  
\end{definition}

\begin{theorem} \label{theorem:kronecker}
The dimension of the Kronecker polytope is the Kronecker coefficient $g_{\lambda \mu\nu}$.
\end{theorem}
\begin{proof}
The tensor product $\mathrm{Specht}(\lambda) \otimes \mathrm{Specht}(\mu) \otimes \mathrm{Specht}(\nu)$ is given by the column span of the matrix with entries $\mathrm{Y}(p, s) \cdot \mathrm{Y}(q, t) \cdot \mathrm{Y}(r, u)$.  The summation over $S_n$ produces $S_n$-invariant vectors.
\qed
\end{proof}


\begin{definition}
If $\lambda, \mu, \nu$ are partitions of $l$, $m$, and $l+m$ respectively, the \emph{Littlewood-Richardson coefficient} $c_{\lambda \mu}^{\nu}$ is defined by
$$
c_{\lambda \mu}^{\nu} = \dim \; (\mathrm{Specht}(\lambda) \boxtimes \mathrm{Specht}(\mu) \otimes \mathrm{Res}^{S_{l+m}}_{S_l \times S_m} \; \mathrm{Specht}(\nu))^{S_l \times S_m},
$$
where $S_l \times S_m$ acts on $\mathrm{Specht}(\lambda) \boxtimes \mathrm{Specht}(\mu)$ separately in the two tensor factors, and $S_l\times S_m$ acts on $\mathrm{Res}^{S_{l+m}}_{S_l \times S_m} \; \mathrm{Specht}(\nu)$ by considering $S_l\times S_m$ as a subgroup of $S_{l+m}$ and using the $S_{l+m}$ action on $\mathrm{Specht}(\nu)$.

We have used the notation $\boxtimes$ for the tensor product with this separated action in order to contrast with the diagonal action that we indicate by $\otimes$.
\end{definition}

\begin{definition}
\label{def:lr-matrix}
The \emph{Littlewood-Richardson matrix} has rows indexed by the product
$$
\{ \mbox{ rearrangements of $x_1$ } \} \times \{ \mbox{ rearrangements of $y_1$ } \} \times \{ \mbox{ rearrangements of $w_1$ } \} 
$$
and columns indexed by the product
$$
\{ \mbox{ rearrangements of $x_2$ } \} \times \{ \mbox{ rearrangements of $y_2$ } \} \times \{ \mbox{ rearrangements of $w_2$ } \}, 
$$
where the $((p, q, r), (s, t, u))$ entry is given by the formula 
$$
\sum_{\sigma \times \tau \in S_l \times S_m} \mathrm{Y}(\sigma s, p) \cdot \mathrm{Y}(\tau t, q) \cdot \mathrm{Y}((\sigma \times \tau) u, r).
$$
Its columns define the \emph{Littlewood-Richardson matroid}.
The convex hull of its columns defines the \emph{Littlewood-Richardson polytope}.  
\end{definition}

Figure~\ref{LittlewoodRichardsonPolytope(2,1)(2,1)(3,2,1)} is a drawing of the Littlewood-Richardson polytope for $\lambda = 2 + 1$, $\mu = 2 + 1$, $\nu = 3 + 2 + 1$.  Since $c_{\lambda \mu}^{\nu} = 2$, this polytope is actually a polygon.
\begin{figure}
\centering
\includegraphics[scale = 2.2]{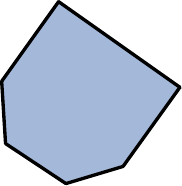}
\caption{Littlewood-Richardson Polytope for $c_{(2,1)(2,1)}^{(3,2,1)} = 2$}
\label{LittlewoodRichardsonPolytope(2,1)(2,1)(3,2,1)}
\end{figure}

\begin{theorem} \label{theorem:littlewood-richardson}
The dimension of the Littlewood-Richardson polytope is the Littlewood-Richardson coefficient $c_{\lambda \mu}^{\nu}$.
\end{theorem}

The proof for this theorem is entirely analogous to that of Theorem~\ref{theorem:kronecker}.

We now study restriction to the \emph{wreath subgroup} $S_l \wr S_m \subseteq S_{lm}$.
Thinking of $lm$ as an $l \times m$ array of dots to be permuted, the wreath subgroup is generated by the permutations for which every dot stays in its row together with the permutations that perform the same operation in every column simultaneously.  Abstractly, the wreath subgroup is isomorphic to the semidirect product $(S_l)^m \rtimes S_m$ where the second factor acts on the first by permuting coordinates.
\begin{definition}
If $\lambda$, $\mu$, and $\nu$ are partitions of $l$, $m$, and $l \cdot m$ respectively, the \emph{plethysm coefficient} $p_{\lambda\mu}^{\nu}$ is defined by
$$
p_{\lambda\mu}^{\nu} = \dim \; (\mathrm{Specht}(\lambda)^{\boxtimes m} \otimes^{\hspace{-1.5pt}\rtimes} \mathrm{Specht}(\mu) \otimes \mathrm{Res}^{S_{lm}}_{S_l \wr S_m} \mathrm{Specht}(\nu))^{S_l \wr S_m},
$$
where $(S_l)^m \rtimes S_m$ acts on $\mathrm{Specht}(\lambda)^{\boxtimes m} \otimes^{\hspace{-1.5pt}\rtimes} \mathrm{Specht}(\mu)$ by $S_m$ on the second factor, and the normal subgroup $(S_l)^m \trianglelefteq (S_l)^m \rtimes S_m$ acts naturally on the first factor.
\end{definition}

As before, let $x_1, x_2$ be a complementary pair of words of length $l$ that correspond via Theorem~\ref{theorem:diagrams} to the partition $\lambda$, and similarly suppose that $y_i$ correspond to $\mu$, and that $w_i$ correspond to $\nu$.

\begin{definition}
\label{def:pleth-matrix}
The \emph{plethysm matrix} has rows indexed by the product
$$
 \{ \mbox{ rearrangements of $x_1$ } \}^m \times \{ \mbox{ rearrangements of $y_1$ } \} \times \{ \mbox{ rearrangements of $w_1$ } \} 
$$
and columns indexed by the product
$$
\{ \mbox{ rearrangements of $x_2$ } \}^m \times \{ \mbox{ rearrangements of $y_2$ } \} \times \{ \mbox{ rearrangements of $w_2$ } \}, 
$$
where the $((\hat{p},q,r), (\hat{s},t,u))$ entry is given by the formula 
$$
\sum_{\hat{\sigma} \rtimes \tau \in S_l \wr S_m} \mathrm{Y}(\hat{\sigma}_1 \hat{s}_1,\hat{p}_1) \cdot \; \cdots \; \cdot \mathrm{Y}(\hat{\sigma}_m \hat{s}_m, \hat{p}_m) \cdot \mathrm{Y}(\tau t,q) \cdot \mathrm{Y}((\hat{\sigma} \rtimes \tau) u,r).
$$
Its columns define the \emph{plethysm matroid}.
The convex hull of its columns defines the \emph{plethytope}.  
\end{definition}

\begin{theorem} \label{theorem:plethysm}
The dimension of the plethytope is the plethysm coefficient $p_{\lambda \mu}^{\nu}$.
\end{theorem}

The proof for this theorem is also analogous to that of Theorem~\ref{theorem:kronecker}.

\section{Computer Calculations} \label{section:code}
The following  \emph{Sage} \cite{Sage} code generates the Specht matrix given a row word and a column word.
\begin{verbatim}
def distinctColumns(w1, w2):
    if len(w2) != len(w2): return False
    seen = set()
    for i in range(len(w1)):
        t = (w1[i], w2[i])
        if t in seen: return False
        seen.add(t)
    return True

def YoungCharacter(w1, w2):
    assert distinctColumns(w1, w2)
    wp = [(w1[i], w2[i]) for i in range(len(w1))]
    def ycfunc(r1, r2):
        if not distinctColumns(r1, r2):
            return 0
        rp = [(r1[i], r2[i]) for i in range(len(w1))]
        po = [wp.index(rx) + 1 for rx in rp]
        return Permutation(po).sign()
    return ycfunc

def SpechtMatrix(w1, w2):
    yc = YoungCharacter(w1, w2)
    mat = []
    for r1 in Permutations(w1):
        row = []
        for r2 in Permutations(w2):
            row = row + [yc(r1, r2)]
        mat = mat + [row]
    return matrix(QQ, mat)

sm22 = SpechtMatrix([1,1,2,2], [1,2,1,2])
print sm22
\end{verbatim}
The output of the code:
\begin{verbatim}
[ 0  1 -1 -1  1  0]
[-1  0  1  1  0 -1]
[ 1 -1  0  0 -1  1]
[ 1 -1  0  0 -1  1]
[-1  0  1  1  0 -1]
[ 0  1 -1 -1  1  0]
\end{verbatim}

Having a Specht matrix, we can use \emph{Macaulay2} \cite{M2} package \emph{Polyhedra} \cite{Polyhedra} to obtain some information about Specht polytopes.
\begin{verbatim}
loadPackage "Polyhedra";
V = matrix{{0,1,-1,-1,1,0},{-1,0,1,1,0,-1},
{1,-1,0,0,-1,1},{1,-1,0,0,-1,1},{-1,0,1,1,0,-1},
{0,1,-1,-1,1,0}}
P = convexHull V
fVector P
\end{verbatim}
The output of the above code, line by line, is:
\begin{verbatim}
      | 0  1  -1 -1 1  0  |
      | -1 0  1  1  0  -1 |
      | 1  -1 0  0  -1 1  |
      | 1  -1 0  0  -1 1  |
      | -1 0  1  1  0  -1 |
      | 0  1  -1 -1 1  0  |
           
                6        6
      Matrix ZZ  <--- ZZ
	
      {ambient dimension => 6           }
       dimension of lineality space => 0
       dimension of polyhedron => 2
       number of facets => 3
       number of rays => 0
       number of vertices => 3
       
       {3, 3, 1}
\end{verbatim}
The following commands give us a description of the faces of co-dimension $i$ and the vertices on each face of a polytope $P$:
\begin{verbatim}
F_i = faces(i,P)
apply(F_i,vertices)
\end{verbatim}
For $i=1$ the output is:
\begin{verbatim}
       {{ambient dimension => 6           },
        dimension of lineality space => 0
        dimension of polyhedron => 1
        number of facets => 2
        number of rays => 0
        number of vertices => 2
        
        {{ambient dimension => 6           },
        dimension of lineality space => 0
        dimension of polyhedron => 1
        number of facets => 2
        number of rays => 0
        number of vertices => 2
        
        {{ambient dimension => 6           }}
        dimension of lineality space => 0
        dimension of polyhedron => 1
        number of facets => 2
        number of rays => 0
        number of vertices => 2
        
      {| -1 1  |, | 0  -1 |, | 0  1  |}
       | 1  0  |  | -1 1  |  | -1 0  |
       | 0  -1 |  | 1  0  |  | 1  -1 |
       | 0  -1 |  | 1  0  |  | 1  -1 |
       | 1  0  |  | -1 1  |  | -1 0  |
       | -1 1  |  | 0  -1 |  | 0  1  |
\end{verbatim}
\label{section:ChowCode}
The following \emph{Sage} code computes the Hilbert Series of the Chow ring for a given matroid. The code computing the Chow ring was contributed to the \emph{Sage} system by Travis Scrimshaw. In the example we use the Specht matrix sm22 computed above.   
\begin{verbatim}
def chow_ring_dimensions(mm, R=None):
    # Setup
    if R is None:
        R = ZZ
    # We only want proper flats
    flats = [X for i in range(1, mm.rank())
             for X in mm.flats(i)]
    E = list(mm.groundset())
    flats_containing = {x: [] for x in E}
    for i,F in enumerate(flats):
        for x in F:
            flats_containing[x].append(i)

    # Create the ambient polynomial ring
    from sage.rings.polynomial\
        .polynomial_ring_constructor 
    import PolynomialRing
    try:
        names = ['A{}'.format(''.join(str(x) 
                 for x in sorted(F)))
                 for F in flats]
        P = PolynomialRing(R, names)
    except ValueError: # variables have
                       # improper names
        P = PolynomialRing(R, 'A', len(flats))
        names = P.variable_names()
    gens = P.gens()
    # Create the ideal of quadratic relations
    Q = [gens[i] * gens[i+j+1]
         for i,F in enumerate(flats)
         for j,G in enumerate(flats[i+1:]) 
         if not (F < G or G < F)]
    # Create the ideal of linear relations
    L = [sum(gens[i] for i in flats_containing[x])
         - sum(gens[i] for i in flats_containing[y])
         for j,x in enumerate(E) for y in E[j+1:]]
    # Compute Hilbert series using Macaulay2
    macaulay2.eval("restart")
    macaulay2.eval("R=QQ[" + str(gens)[1:-1] + "]")
    macaulay2.eval("I=ideal(" + str(Q)[1:-1] + ",
    " + str(L)[1:-1] + ")")
    hs = macaulay2.eval("toString hilbertSeries I")
    T = PolynomialRing(RationalField(),"T").gen()
    return sage_eval(hs, locals={'T':T})
    
chow_ring_dimensions(Matroid(sm22))
\end{verbatim}
The output of the code for our example is
\begin{verbatim}
T+1
\end{verbatim}

We now give the code for Examples~\ref{Example:apprenticeship1} and~\ref{Example:apprenticeship2}.
The following \emph{Sage} commands compute the matroid corresponding to a given matrix, the lattice of flats of a matroid, a list of flats of a given rank, and the characteristic polynomial of a matroid.
\begin{verbatim}
X = matrix([[1, 0, 0, 1, 1, 1], [0, 1, 0, 2, 3, 4],
			[0, 0, 1, 0, 0, 1]])
M = Matroid(X)
M
M.lattice_of_flats()
sorted([sorted(F) for F in M.lattice_of_flats()])
F1 = M.flats(1)
sorted([sorted(F) for F in F1])
rank = M.rank()
Tutte_polynomial = M.tutte_polynomial()
Tutte_polynomial
var('t')
char_poly = (-1)^rank * expand(Tutte_polynomial(1-t,0))
char_poly
\end{verbatim}
The output of the above code is:
\begin{verbatim}
Linear matroid of rank 3 on 6 elements represented over
the Rational Field
Finite lattice containing 18 elements
[[], [0], [0, 1, 2, 3, 4, 5], [0, 1, 3, 4], [0, 2],
[0, 5], [1], [1, 2], [1, 5], [2], [2, 3], [2, 4],
[2, 5], [3], [3, 5], [4], [4, 5], [5]]
[[0], [1], [2], [3], [4], [5]]
x^3 + x*y^2 + y^3 + 3*x^2 + 2*x*y + 2*y^2 + 3*x + 3*y
t
t^3 - 6*t^2 + 12*t - 7
\end{verbatim}

\begin{acknowledgement}
This article was initiated during the Apprenticeship Weeks (22 August-2 September 2016), led by Bernd Sturmfels, as part of the Combinatorial Algebraic Geometry Semester at the Fields Institute.
The authors wish to thank Bernd Sturmfels, Diane Maclagan, Gregory G. Smith for their leadership and encouragement, and all the participants of the Combinatorial Algebraic Geometry thematic program at the Fields Institute, at which this work was conceived.  They also thank the Fields Institute and the Clay Mathematics Institute for hospitality and support.  Finally, thanks to Bernd Sturmfels also for suggesting the term ``matroidification'' and to anonymous referees for their helpful suggestions.
\end{acknowledgement}

\end{document}